\documentclass[twoside,12pt]{article}
\usepackage{amsthm,amsmath,amssymb,amscd,enumerate,epsfig}
\usepackage{amsfonts}
\usepackage{pst-all,fancybox}
\usepackage{graphicx}
\usepackage{fancyhdr}
\usepackage{enumitem}
\usepackage{multirow}
\begin{document}
\textheight=570pt 
\begin{center} {\Large\bf  On Convergence of an Accelerated Modified Newton Method for Nonlinear Equations}
\vskip 1cm {\bf  Sanwar Ahmad$^1$, Joy Watson$^2$, Mohammad Tabanjeh$^3$}\\
\
\\Department of Mathematics and Economics, \\
College of Engineering and Technology\\
Virginia State University\\
Petersburg, USA\\

\
\\email: $^1$sahmad@vsu.edu, $^2$jwat2573@students.vsu.edu, $^3$MTabanjeh@vsu.edu\\
\end{center}
\begin{abstract}
Newton's iteration is a fundamental tool for root-finding and numerical solutions of systems of equations. The iteration rapidly refines the initial approximation to the exact root, and in general the convergence is quadratic. Since the method requires finding the function value and its derivative at each iteration, in some cases, it may not converge. This is because the value of the derivative gets close to zero. In this paper, we introduce a modified and stable algorithm of Newton's iteration method that addresses this issue, reduces computational cost, and improves efficiency. In addition. we analyze the convergence properties of the modified method to demonstrate its effectiveness.
\end{abstract}
\begingroup
\newtheorem{theorem}{Theorem}[section]
\newtheorem{lemma}[theorem]{Lemma}
\newtheorem{proposition}[theorem]{Proposition}
\newtheorem{corollary}[theorem]{Corollary}
\newtheorem{definition}[theorem]{Definition}
\newtheorem{remark}[theorem]{Remark}
\endgroup
\def\E#1{\left\langle #1 \right\rangle}
\def\diag{\mathop{\fam 0\relax diag}\nolimits}
\section{Introduction}
Root-finding algorithms for solving the equation $f(x)=0$ are ubiquitous in computational mathematics and serve as essential tools in scientific and engineering applications. Among these, Newton's method \cite{ascher2011first,kincaid1991sci} stands out for its simplicity and rapid quadratic convergence to the exact root. 
Given an initial approximation $x_0$ to the exact root, the classical Newton method for a nonlinear equation $f(x)=0$ can be written as
\begin{align}
    x_{k+1} = x_k - \frac{f(x_k)}{f'(x_k)}, \quad \mbox{for } k = 0,1,2,... \label{original}
\end{align}

In recent years, numerous variations of Newton's method with higher-order convergence have been developed and analyzed for solving nonlinear equations. \cite{mcdougall2014simple} modified the method with a convergence of order $(1+\sqrt{2})$, whereas \cite{homeier2003modified,scavo1995geometry,weerakoon2000variant} introduced cubic convergence method. All of these methods are predictor-corrector-type methods that require more evaluation steps of the function and its derivative at each iteration than the classical method. This might become computationally intensive and costly. However, the additional cost can be offset by the higher rate of convergence. 

This motivates us to explore a simple modification of Newton's method, as introduced in \cite{hunter2001applied}. In this paper, we analyze the convergence properties of the modified method, derive the convergence criterion, and investigate the rate of convergence through several numerical examples. One key advantage of this modified method is that it eliminates the need for direct computation of derivatives at each iteration, thus reducing computational overhead and improving numerical stability.

\section{Preliminary}
\begin{definition}
A sequence $\{x_k: k \in \mathbb{N}\}$ of real numbers is said to converge to $\alpha \in \mathbb{R}$ if 
\[ \lim_{k \to \infty} |x_k-\alpha| = 0. \]
If, in addition, there exist constants $c\ge 0$ and $p\ge 0$ such that for all $k \ge K$,
\[ \left|x_{k+1} - \alpha \right| \le c \left|x_{k} - \alpha \right|^p \]
for some nonnegative integer $K$, then $\{x_k\}$ is said to converge to $\alpha$ with $q$-order at least $p$, \cite{weerakoon2000variant}.
\end{definition}
\begin{definition}
    The rate of convergence represents how quickly an iterative method approaches the final solution; it can be estimated as 
\begin{align}\label{rc}
    q \approx \frac{\log\|\frac{x_{k+1}-x_k}{x_k-x_{k-1}}\|}{\log\|\frac{x_{k}-x_{k-1}}{x_{k-1}-x_{k-2}}\|}. 
\end{align}
\end{definition}
Consider the problem of finding a simple root $\alpha$ of the equation, 
\begin{align}\label{equ}
    f(x) = 0.
\end{align}
Newton's method is a well-known iterative method that uses the following iteration.
\begin{align}\label{nm}
    x_{k+1} = x_k - \frac{f(x_k)}{f'(x_k)}, \mbox{ for } k = 0,1,2,...
\end{align}
provided that $f'(x_k) \ne 0$ for any $k$ and $x_0$ is chosen sufficiently close to the root $\alpha$.

Suppose that we have already computed $x_k$, note that the error in computing $x_k$ is $e_k = |x_k-\alpha|$.  To derive a formula that relates the error after the next step,  $e_{k+1} = |x_{k+1}-\alpha|$  to  $e_k$, we can use Taylor's expansion of $f(x)$ about $x=x_k$, that is;
\begin{align}
    f(x)=f(x_k) + f'(x_k)(x-x_k) + \frac{1}{2}f''(c)(x-x_k)^2
\end{align}
for some $c$  between  $x$  and  $x_k$.  In particular, if we choose $x=\alpha$, then 

\begin{align}\label{TS}
    0= f(x_k) + f'(x_k)(\alpha-x_k) + \frac{1}{2}f''(c)(\alpha-x_k)^2.
\end{align} 
Recall that $x_{k+1}$  is the solution of $0=f(x_k) + f'(x_k)(x_{k+1}-x_k)$. Now, subtract the last equation from (\ref{TS}) to get;
\begin{align*}
    0 &= f'(x_k)(\alpha-x_{k+1}) + \frac{1}{2}f''(c)(\alpha-x_k)^2 \\
    \implies |x_{k+1}-\alpha| &= \frac{|f''(c)|}{2|f'(x_k)|} |x_k-\alpha|^2. 
\end{align*}
If the guess $x_k$  is close to $\alpha$,  then  $c$,  which must be between $x_k$  and $\alpha$,  is also close to $\alpha$. Thus, $f''(c) \approx f''(\alpha)$ and  $f'(x_k) \approx f'(\alpha)$  and hence
\[ |x_{k+1}-\alpha| \approx \frac{|f''(\alpha)|}{2|f'(\alpha)|} |x_k-\alpha|^2. \]
Therefore, we can conclude that if $x_0$ is chosen sufficiently close to the simple root $\alpha$, then Newton's method that is defined in (\ref{nm}) will converge quadratically \cite{ascher2011first,kincaid1991sci}.
\section{Modified Newton Method}
In this section, we assume that $y = f(x)$ is differentiable and has at least one real root within some interval $[a,b]$. Newton's method for solving the nonlinear equation $f(x)=0$ is given by 
\begin{align}
    x_{k+1} = x_k - \frac{f(x_k)}{f'(x_k)}, \quad \mbox{for }k = 0,1,2,... \label{modified}
\end{align}
which relies on evaluating the derivative $f^{\prime}(x_k)$ at every iteration. However, this method fails if $f^{\prime}(x_k)=0$ at any iteration $k$, and it may be computationally expensive or impractical to evaluate the derivative at each step, particularly if no analytical expression is available.

To mitigate these issues, we adopt a modified version of Newton's method in which the derivative is calculated once at the initial point $x_0\in(a,b)$, where $f'(x_0)\ne 0$, and is used throughout the iteration process. This results in the simplified iteration formula:
\begin{align}
    x_{k+1} = x_k - \frac{f(x_k)}{f'(x_0)}, \quad \mbox{for }k = 0,1,2,... \label{modified}
\end{align}
This approach avoids division by zero at later iterations and reduces the computational cost by requiring only one derivative evaluation, compared to classical Newton's method, which requires one per iteration. As noted by Hunter \cite{hunter2001applied}, this significantly reduces the number of operations required and leads to faster evaluations when function evaluations are inexpensive but derivatives are costly. Geometrically, the update step in this modified method corresponds to finding the zero of the line through the point $(x_k,f(x_k))$, which is parallel to the tangent line to $f(x)$ at the initial point $x=x_0$.

Another computing advantage is that the derivative $f'(x)$ only has to be computed and inverted once at the starting point \cite{hunter2001applied}. Hence, the number of function evaluations at every iteration will be reduced at least by halves compared to the classical Newton iteration or other modified methods. 
\subsection{Convergence Criterion}
The following theorem (see, \cite{hunter2001applied}), proves the convergence of the modified Newton's method (\ref{modified}).
\begin{theorem}
Let \( f: U \subset X \to Y \) be a differentiable map from an open subset \( U \) of a Banach space \( X \) into a Banach space \( Y \), such that \( f' \) is Lipschitz continuous in \( U \) with Lipschitz constant \( C \). Suppose that \( x_0 \in U \), \( f'(x_0) \ne 0\) and
\[
h = C \left\| \left[ f'(x_0) \right]^{-1} \right\| \left\| \left[ f'(x_0) \right]^{-1} f(x_0) \right\| \leq \frac{1}{4}.
\] Define
\[
\delta = \left\| \left[ f'(x_0) \right]^{-1} f(x_0) \right\| \left( \frac{1 - \sqrt{1 - 4h}}{2h} \right) \geq 0,
\]
and suppose further that the closed ball of radius $\delta$ centered at $x_0$ is contained in $U$ ($B_\delta \subset U$), that is, $$ B_\delta(x_0) = \{x \in X : \|x-x_0\| \leq \delta\}.$$
Then, there is a unique solution of the equation $f(x)=0$ in $B_\delta$, and the sequence $(x_k)$ of modified Newton's iteration defined by (\ref{modified}) converges to a unique solution in $B_\delta$ as  $k \rightarrow \infty$.
\end{theorem}
{\bf Proof.}
The modified Newton iterates are obtained from the fixed-point iteration \( x_{k+1} = T(x_k) \), where
\[
T(x) = x - \left[ f'(x_1) \right]^{-1} f(x).
\]
First, we show that \( T: B_\delta \to B_\delta \). We may write
\[
T(x) - x_1 = - \left[ f'(x_1) \right]^{-1} \left[ r(x) + f(x_1) \right],
\]
where
\[
r(x) = f(x) - f(x_1) - f'(x_1)(x - x_1).
\]
Taking the norm, we find that
\[
\| T(x) - x_1 \| \leq M \| r(x) \| + \eta,
\]
where
\[
M = \left\| \left[ f'(x_1) \right]^{-1} \right\|,
\quad \eta = \left\| \left[ f'(x_1) \right]^{-1} f(x_1) \right\|.
\]
Now, we compute the derivative of \( r \), and use the Lipschitz condition for \( f' \) to obtain 
\[
\| r'(x) \| = \| f'(x) - f'(x_1) \| \leq C \| x - x_1 \|.
\]
Since \( r(x_1) = 0 \), the mean value theorem implies that
\[
\| r(x) \| = \| r(x) - r(x_1) \| \leq \sup_{0 \leq t \leq 1} \| r'(tx + (1 - t)x_1) \| \| x - x_1 \| \leq C \| x - x_1 \|^2.
\]
Using this result in the norm taken above, we find that
\[
\| T(x) - x_1 \| \leq M||r(x)||+\eta \le CM \| x - x_1 \|^2 + \eta.
\]
Hence, \( T \) maps the ball \( \{ x : \| x - x_1 \| \leq \epsilon \} \) into itself provided that
\[
CM \epsilon^2 + \eta \leq \epsilon.
\]
This inequality can be satisfied for some \( \epsilon > 0 \) if
\[
h = CM \eta \leq \frac{1}{4}.
\]
 Using the definition of $M$ and $\eta$, we see that this is the condition in the theorem above. In this case, the smallest value of $\delta$ and $\epsilon$ for which $CM\epsilon^2+\eta\le \epsilon$  holds is  $$\delta = \eta \tau $$ where $\tau$ is the smallest root of the equation $h\tau^2-\tau+1=0$, or $$\tau = \frac{1-\sqrt{1-4h}}{2h}.$$
Substituting $\eta$ and $\tau$ in the equation $\delta=\eta \tau$,  we find that $\delta$ is given by 
$$\delta=||[f^{\prime}(x_1)]^{-1}f(x_1)||(\frac{1-\sqrt{1-4h}}{2h})
.$$ 
This proves that $T: B_\delta \rightarrow B_\delta.$

Next, we prove that $T$ is a contraction on $B_\delta$. Differentiate
\[
T(x) - x_1 = - \left[ f'(x_1) \right]^{-1} \left[ r(x) + f(x_1) \right],
\]
yield $$T'(x) = -[f'(x_1)]^{-1} [f'(x)-f'(x_1)].$$
Hence using the definition of $M$ and $\eta$ and the Lipschitz condition on $f'$, we have $$||T'(x)|| \leq M||f'(x)-f'(x_1)||\leq CM||x-x_1||\leq CM\delta, \quad \forall x \in B_\delta.$$
It follows from $h, \delta, \tau$ that $$CM\delta = \frac{1-\sqrt{1-4h}}{2}\leq\frac{1}{2}.$$
Therefore, we have $||T'(x)||\leq 1/2$ in $B_\delta$, so from the mean value theorem $$||T(x)-T(y)||\leq\frac{1}{2}||x-y|| ~~\forall x,y \in B_\delta. $$ 
The theorem now follows the contraction mapping theorem.
\begin{corollary} 
Using the proof of the above theorem, we can claim the following: 
\begin{enumerate}[label = (\alph*)]
    \item The convergence of the modified method is at least linear.
    \item Under the assumptions of the above theorem, the convergence is quadratic.
\end{enumerate}
\end{corollary}
\noindent
\textbf{Proof.}  (a) Using 
\[
T(x) = x - \left[ f'(x_1) \right]^{-1} f(x).
\]
If $x=\alpha$ is the exact root of the equation $f(x)=0$, we may write
\[
T(x) - \alpha = - \left[ f'(x_1) \right]^{-1} \left[ r_1(x) \right],
\]
where
\[
r_1(x) = f(x) - f'(x_1)(x - \alpha).
\]
Taking the norm, we find that
\[
\| T(x) - x_1 \| \leq M \| r(x) \|,
\]
where
\[
M = \left\| \left[ f'(x_1) \right]^{-1} \right\|.
\]
Since \( r_1(\alpha) = 0 \), the mean value theorem implies that
\[
\| r_1(x) \| = \| r_1(x) - r_1(\alpha) \| \leq \left|r'_1(\zeta)\right|  \left| x - \alpha \right|.
\]
Using this result in the norm taken above, we find that
\[
\left| x_{k+1}-\alpha \right| = \left| T(x_k) - \alpha \right| \leq  M \left|r'_1(\zeta)\right| \left| x_k - \alpha \right|. \]
Thus, the rate of convergence is at least linear. \\ \\

(b) Using the proof of the theorem and if the assumptions are true, it is easy to show that,
\[
\left| x_{k+1}-\alpha \right| = \left| T(x_k) - \alpha \right| \leq M C  \left| x - \alpha \right|^2,
\]
where $C$ is the Lipschitz constant. \qed

The assumptions $x_0 \in U$, $h \leq 1/4$, $\delta \geq 0$, and $B_\delta(x_0) \subset U$ in the theorem hypotheses are satisfied when $x_0$ is sufficiently close to a solution of $f(x)=0$ at which the derivative of $f$ is nonsingular, \cite{hunter2001applied}.

Since the derivative remains constant across the iterations, the convergence is based on the choice of initial guess. With $\displaystyle g(x)=x-\frac{f(x)}{f'(x_0)}$, the modified Newton method algorithm (\ref{modified}) can be viewed as finding the fixed point of the function $y=g(x)$, that is, $$x=g(x).$$
We see that $f'(x_0) \neq 0,$ and it will converge if $|g'(x)|<1.$ Taking the derivative of both sides and using the convergence of the fixed point iterative method, we need
\begin{align}
  |g'(x)|=\left|1-\frac{f'(x)}{f'(x_0)}\right| < 1  \nonumber \\
  \implies \frac{1}{2} \left|f'(x)\right|  < \left|f'(x_0)\right|. \label{conv_criterion}  
\end{align}
From (\ref{conv_criterion}) and using the convergence of the fixed-point iteration method \cite{ascher2011first,kincaid1991sci}, we observe that the convergence of the modified method (\ref{modified}) depends on the choice of $x_0$ and will converge to the root quadratically if $f'(x_0) = f'(x)$ when $x$ is closer to the root. However, if the tangent line of $y = f(x)$ at its zero is nearly vertical, then the method may converge to the root very slowly, or even fail to converge. Moreover, for functions with repeated roots, the modified method will show a very slow convergence. 
\section{Numerical Examples.}

In this section, we evaluate and compare the performance of the proposed accelerated (modified Newton) method against the classical Newton method for solving nonlinear equations. All computations were performed using MATLAB 2023a on an iMac with 8 GB RAM and an Apple M1 Chip running macOS Sonoma 14.4.1. The convergence rate was estimated using equation (\ref{rc}), and a tolerance of $10^{-15}$ was used as the stopping criterion for all experiments.

Table \ref{tab:comparison} summarizes the results for five test functions. The functions are listed as, (i)  $f(x) = x^3-3x^2+3x-2$, (ii) $f(x) = xe^{x^2}-\sin^2(x)+3\cos(x)+5$, (iii)  $f(x) = \sin^2(x)-x^2+1$, (iv) $f(x) = x^3+4x^2-10$ and (v) $f(x) = x^2-e^x-3x+2$. We observe that the convergence behavior of the modified method ranges from linear to quadratic, depending on the choice of initial guess $x_0$. When quadratic convergence is attained, the modified method consistently outperforms the classical Newton method in terms of execution time. Although both methods require the same number of iterations in these cases, the reduced computational cost of evaluating the derivative only once gives the modified method a clear advantage.

In cases where the modified method converges superlinearly or linearly, it may require slightly more iterations than the classical method. However, it still often achieves faster overall runtime. This efficiency is particularly noticeable in problems where derivative evaluations are costly.

These results highlight the effectiveness of the modified Newton method as a computationally efficient alternative, especially in contexts where derivative evaluations dominate the computational cost.

\begin{table}[!htb]
    \centering
\begin{tabular}{c|c|c|c|c|c}
\hline $x_0$  & Method & Root & $N$ & Comp. time (s)& Rate of Conv. \\
\hline   0 & Modified & 2 & 10 & 0.018562 & 2.0034\\ 
\cline{2-2} \cline{3-3} \cline{4-4} \cline{5-5} \cline{6-6} &NM & 2 & 10 & 0.021913 & 2.0008\\ \hline \hline  -1 & Modified & -1.20764 & 117 & 0.015107 & 1.40942\\ 
\cline{2-2} \cline{3-3} \cline{4-4} \cline{5-5} \cline{6-6} &NM & -1.20764 & 7 & 0.026856 & 1.88291\\ \hline
\hline   1.6 & Modified & 1.40449 & 24 & 0.020830 & 1.03647\\ 
\cline{2-2} \cline{3-3} \cline{4-4} \cline{5-5} \cline{6-6} &NM & 1.40449 & 6 & 0.023484 & 2.0007\\ \hline \hline
 1.5 & Modified & 1.36523 & 17 & 0.020378 & 1.10212\\
\cline{2-2} \cline{3-3} \cline{4-4} \cline{5-5} \cline{6-6} &NM & 1.36523 & 5 & 0.016596 & 2.0004\\ \hline \hline
 1.5 & Modified & 0.25753 & 21 & 0.025920 & 0.93483\\
\cline{2-2} \cline{3-3} \cline{4-4} \cline{5-5} \cline{6-6} &NM & 0.25753 & 6 & 0.028125 & 2.0005\\ \hline \hline
\end{tabular}   
    \caption{Comparison of computing time and rate of convergence}
    \label{tab:comparison}
\end{table}
\section*{Conclusion}
In this article, we analyze a modified Newton method for solving non-linear equations, designed to reduce computational cost by minimizing the number of derivative evaluations. We presented a theoretical convergence analysis establishing that the method generally exhibits linear convergence but can achieve quadratic convergence under certain conditions-particularly with suitable initial guesses and sufficient smoothness of the function. These findings were supported by numerical examples.

Our computational experiments demonstrated that the modified method is often more efficient than the classical Newton method in terms of running time, even when it requires more iterations. 
This efficiency arises from the fact that the derivative is evaluated only once, making the method especially appealing for problems where derivative computation is expensive or analytically intractable.

Overall, the modified Newton method offers a practical and efficient alternative to the classical approach. It provides robust performance across a variety of nonlinear problems. Future research may focus on extending this approach to systems of nonlinear equations, developing adaptive strategies for selectively updating the derivative, and applying the method to real-world problems in applied mathematics, physics, and engineering domain where computational efficiency is critical.

\end{document}